\documentclass[11pt,reqno,a4paper]{amsart}

\usepackage{todonotes}
\usepackage{amsmath}
\usepackage{amsfonts}
\usepackage{mathrsfs}  
\usepackage{amssymb}
\usepackage{amsthm}
\usepackage{hyperref}
\usepackage{graphicx}
\usepackage[noadjust]{cite}
\usepackage{caption}
\usepackage{dutchcal}
\usepackage{bbm}
\usepackage{tikz}
\usetikzlibrary{decorations.pathreplacing}
\usepackage[T1]{fontenc}
\usepackage[utf8]{inputenc}
\usepackage{mathtools}

\usepackage[top=3.5cm, bottom=3.5cm, left=2.7cm, right=2.7cm, headsep=0.3in]{geometry}

\usepackage{fancyhdr}

\theoremstyle{plain}
\newtheorem{theorem}{Theorem}[section]
\newtheorem{defi}[theorem]{Definition}
\newtheorem{corollary}[theorem]{Corollary}
\newtheorem{lemma}[theorem]{Lemma}
\newtheorem{remark}[theorem]{Remark}
\newtheorem{proposition}[theorem]{Proposition}

\newtheorem{que}[theorem]{Question}

\newcommand{\RR} {\mathbb R}

\newcommand{\NN} {\mathbb N}

\newcommand{\NNN}{\mathcal{N}}

\newcommand{\GGG}{\mathcal{G}}
\newcommand{\PPP}{\mathcal{P}}
\newcommand{\DDD}{\mathcal{D}}
\newcommand{\BBB}{\mathcal{B}}

\newcommand{\beq} {\begin{equation}}
\newcommand{\eeq} {\end{equation}}

\numberwithin{equation}{section}

\begin{document}

\title{Sharp lower bound on the ``number of nodal decomposition'' of  graphs}
\author{Hiranya Kishore Dey}
\address{Indian Institute of Science, Bangalore, Karnataka 560012, India}
\email{hiranya.dey@gmail.com} 
\author{Soumyajit Saha}
\address{Iowa State University, Ames, Iowa 50011, USA}
\email{ssaha1@iastate.edu}
\maketitle
\begin{abstract}
    Urschel in \cite{U} introduced a notion of nodal partitioning to prove an upper bound on the number of nodal decomposition of discrete Laplacian eigenvectors. The result is an analogue to the well-known Courant's nodal domain theorem on continuous Laplacian.  In this article, using the same notion of partitioning,  we discuss the lower bound (or lack thereof) on the number of nodal decomposition of eigenvectors in the class of all graphs with a fixed number of vertices (however large). This can be treated as a discrete analogue to the results of Stern and Lewy in the continuous Laplacian case.
\end{abstract}

\section{Introduction and motivation}\label{sec: intro}

We start by recalling that given a connected undirected simple graph $G=(V, E)$ with vertex set $V$ and edge set $E$, the Laplacian of the graph $G$ is a $n\times n$ ($\#V=n$) matrix defined as 
\begin{equation}
    L(G)=  D(G)-A(G) ,
\end{equation}
where $A(G)$ denotes the adjacency matrix of $G$ and $D(G)$ is a diagonal matrix whose entries are the degrees of the vertices of $G$. Then, there are $n$ eigenvalues of $L(G)$
\[0=\lambda_1<\lambda_2\leq \cdots\leq \lambda_n\]
and the corresponding eigenvectors $f_{\lambda_k}$ satisfy the equation
\[L(G)f_{\lambda_k}= \lambda_kf_{\lambda_k}.\]
In this article, we will look into the lower bounds on the number of nodal decomposition of certain graphs. Our article is motivated by several well-known results on the continuous Laplacian, which we discuss below.

Given a Riemannian manifold $(M, g)$, the Laplace-Beltrami operator on $(M, g)$ has a discrete spectrum 
$$\eta_1< \eta_2 \leq \cdots\leq \eta_k\leq \cdots \nearrow \infty$$
with smooth real-valued eigenfunctions satisfying 
$$-\Delta_g \varphi_k=\eta_k \varphi_k.$$
Given an eigenfunction $\varphi$, we say that 
$\NNN_\varphi=\{x\in M: \varphi(x)=0\}$
is the nodal set corresponding to $\varphi$. Note that $\NNN_\varphi$ divides $M$ into several connected components. Each connected component of $M\setminus \NNN_\varphi$ is referred to as a nodal domain corresponding to $\varphi$. The eigenfunction is non-sign-changing in these nodal domains. %Various aspects of the eigenvalues and eigenfunctions have been studied extensively for years. 
A natural question to ask is: \emph{given an eigenfunction, how many nodal domains can we have corresponding to that eigenfunction?} An answer to this is the following nodal domain theorem by Courant, and it is one of the very few global results regarding eigenfunctions and eigenvalues.
\begin{theorem}\label{thm: Courant nodal dom}
Let $\varphi_k$ be the $k$-th eigenfunction and $\nu(\varphi_k)$ be the number of nodal domains corresponding to $\varphi_k$. Then  $\nu(\varphi_k)\leq k$. 
\end{theorem}
Pleijel, in \cite{Pl}, showed that in the Dirichlet case for domains in $\RR^2$, the maximal division by nodal lines could occur for only finitely many eigenfunctions. The result was extended by Peetre in \cite{Pe} to some domains on two-dimensional Riemannian manifolds, and a general result for $n$-dimensional Riemannian compact manifolds was proved by B\'{e}rard and Meyer in \cite{BM}. In the Neumann case, Polterovich in \cite{Po} proved the same for two-dimensional domains with quite regular boundaries, and this was generalised by L\'{e}na in \cite{Le} for higher dimensional domains with $C^{1,1}$ boundary.

Following Courant's nodal domain theorem, the next question one can ask is: \emph{is there a lower bound estimate on the number of nodal domains?} We note here that except for the first eigenfunction, every other eigenfunction should have at least two nodal domains, a trivial lower bound. This follows from the orthogonality of the eigenfunctions and the fact that the first eigenfunction has a constant sign. Combining this fact with the Courant nodal domain theorem, the second eigenfunction always has exactly two nodal domains. Studying the second eigenfunction with Dirichlet and Neumann boundaries has been of special interest over the past few decades, and we refer our readers to \cite{MS, MS1} (and references therein) for more details. Coming back to nodal domain counts of eigenfunctions, the more precise question to ask is: \emph{is there a non-trivial lower bound on the number of nodal domains for the higher eigenfunctions?} In this direction, we look at the following result of Stern from her thesis \cite{Ste}\footnote{For a chronology and proper accreditation of the results of Lewy and Stern, see the discussion in the paper \cite{BH}. }.
\begin{theorem}[Stern]\label{thm: Stern}
For the square $[0, \pi]\times [0,\pi]\subset\RR^2$, there exists a sequence $\{\varphi_m\}$ of Dirichlet eigenfunctions associated with the eigenvalues $\eta_{2m,1}=4m^2+1$, $m\geq 1$, such that $\varphi_m$ has exactly two nodal domains.
\end{theorem}
On surfaces, Lewy in \cite{L} proved the following lower estimate on the number of domains in which the nodal lines of spherical harmonics divide the sphere.
\begin{theorem}[Lewy]\label{thm:Lewy}
Let $k \in \NN$ be odd. Then there is a spherical harmonic of degree $k$ with exactly two nodal domains. Let $k \in \NN$ be even. Then there is a spherical harmonic of degree $k$ with exactly three nodal domains.
\end{theorem}
Both the above results tell us that on the class of planar domains (and surfaces), for certain values of $k\in \NN$ (unlike the upper bound), $\varphi_k$ cannot have a non-trivial lower bound (depending on $k$) on $\nu(\varphi_k)$ in that class.  In this article, we will show an analogous result for the discrete Laplacian. %\emph{Given any graph with $N$(however large) vertices, is it possible to get a non-trivial lower bound on the number of nodal decomposition (in some sense)?}. \mnote{This is repetitive. Plus, you might want to talk about a lower bound that is saturated in some situations. I guess this is precisely why the title has the word ``sharp''.}

Interestingly, spectral partitioning is a well-studied topic in graph theory as well, and in this regard, counting the number of nodal domains has been of special interest over the past few decades. Colin de Verdiere \cite{Co} and Friedman \cite{Fr} mentioned nodal domain-type theorems for graphs, and later Davies et al. \cite{DGLS} proved the first nodal domain-type theorem for graphs. Several results have also been proved by Bıyıko\u{g}lu et al. in \cite{BHLPS} and Gladwell and Zhu in \cite{GZ} on the bounds of the number of strong/weak nodal domains (see Definition \ref{def: strong/weak nodal domain}). Looking at the upper bounds on the number of weak/strong nodal domains, the above results show that Courant's upper bound does not always hold for graphs. Instead of working with the strong and weak nodal domains, Urschel in \cite{U} used a different notion of nodal decomposition of graphs (discussed in Section \ref{sec: notations}) to show a result analogous (see Theorem \ref{Urschel main theorem} below) to the Courant's nodal domain theorem.

Studying the nodal domain count, Bıyıko\u{g}lu in \cite{B} gave a lower bound on the number of nodal domains of trees based on the work of Fiedler in \cite{Fi}. In \cite{Be}, Berkolaiko further generalised the bounds of  Bıyıko\u{g}lu for discrete and metric graphs in terms of the number of links (the minimum number of edges to be removed from the graph to make it a tree). In all these results, the simplicity of the Laplace spectrum is an important assumption and Berkolaiko additionally assumed that the corresponding eigenvectors do not have any zero component.
\iffalse
We mention the below.
\begin{theorem}[Berkolaiko]
    Let $\lambda_k$ be a simple eigenvalue of $L(\GGG)$ and the corresponding eigenvector $f_{\lambda_k}$ is non-zero on each vertex. Then the nodal domain count is bounded below by
    $$k-l\leq \nu(f_{\lambda_k}), $$
    where $l$ is the dimension of the cycle space of $\GGG$.
\end{theorem}
\fi
If an analogous result inspired by Stern and Lewy is to succeed, one should look for graphs that have a highly repeated spectrum.
%A quick glance at the results of Stern and Lewy suggests that any graph with analogous properties should have a highly repeated spectrum. 
%A quick glance at the results of Stern and Lewy suggests that any hope of finding a discrete graph with analogous properties should enforce the graph to have a highly repeated spectrum. 
So, in order to relax the restrictions of spectrum simplicity and non-zero eigenvector components, in this article, we will adopt the notion of nodal partitioning used by Urschel and talk about the lower bound of such nodal decomposition.

\emph{Overview of the paper:} In Section \ref{sec: notations}, we look at the definition of strong/weak nodal domains corresponding to an eigenvector, outline the general description of the nodal decomposition used by Urschel in \cite{U}, and state our main results along with the required notations. In Section \ref{sec: main result}, we provide the proofs of our main results. The proofs are based on an explicit computation of the eigenvectors and eigenvalues of certain graphs. In Section \ref{sec: power graph}, we discuss the nodal decomposition of power graphs (see Definition \ref{defi: power graph}) and end the article by characterising abelian $p$-groups in terms of their nodal decomposition.

\section{Definitions, notations and main results}\label{sec: notations}
Similar to the continuous Laplacian case, we can find the nodal edges and nodal domains from a given eigenvector $f$ of $L(G)$. We jot down their definitions below.

\begin{defi}[Nodal domains]\label{def: strong/weak nodal domain}
    A strong (respectively, weak) nodal domain of a graph $G$ with respect to an eigenvector $f$ is a maximally connected subgraph $H$ satisfying $f(i)f(j) > 0$ (respectively, $f(i)f(j) \geq 0 $) for all $i, j \in V(H)$.
\end{defi}

 The number of strong (resp. weak) nodal domains of a graph $\GGG$ with respect to $f$ is
denoted by $S(f) $ (respectively $W(f)$). It is clear that when the set of nodal vertices $\{u|f(u) = 0\}$ is empty, the definition of a weak and strong nodal domain is equivalent, and $S(f) = W(f)$.

\begin{defi}[Nodal edges]\label{defi: nodal edge}
    For any edge $e_{ij}\in E(G)$ (an edge between vertices $i$ and $j\in V(G)$), we say that it is a strong (respectively, weak) nodal edge  corresponding to eigenvector $f$ if $f(i)f(j)<0$ (respectively, $f(i)f(j)\leq 0$).
\end{defi}

\iffalse
\begin{defi}[Nodal vertices]\label{defi: nodal vertices}
    The set $\{1\leq i \leq n: f(i)=0 \}$ is referred to as the nodal vertices corresponding to $f$. 
\end{defi}
\fi
Ideally, we would want the nodal edges to partition the graph into nodal domains as the nodal set partitions the domain in the continuous Laplacian case. Also, we would want the nodal edges to be disjoint from the nodal domains.  Note that a weak nodal edge might be a part of some weak nodal domain, which prevents it from forming such a partition of our graph. Considering strong nodal edges, we get a partition of the graph into strong nodal domains only when the set $\{1\leq i \leq n: f(i)=0 \}$ is empty. But that might not always be the case. 

Consider the following graph.

 \begin{figure}[ht]
\centering
\includegraphics[scale=0.2]{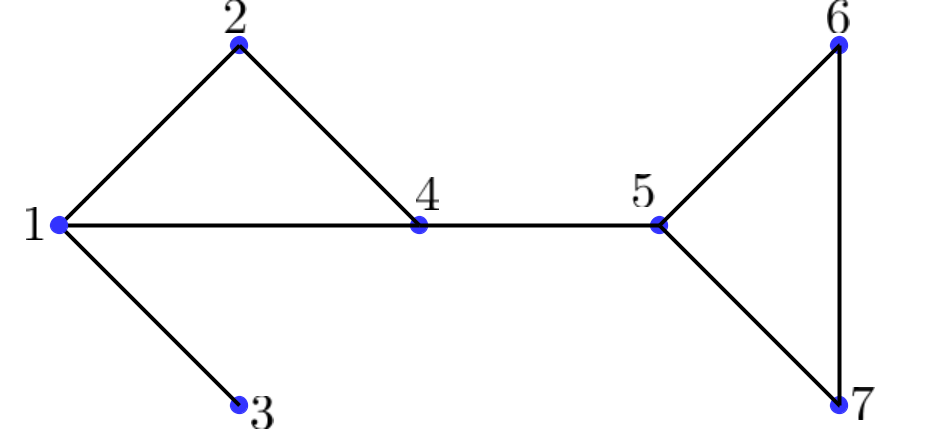}
\caption{Example of a graph}
%\label{fig:Chladni Smithsonian}
\end{figure}

For the above graph, $(0, 1, 0, -1, -1, 1, 0)^T$ is an eigenvector corresponding to the eigenvalue $\lambda=3$. Figure \ref{fig: strong nodal} below shows the strong nodal domains and nodal edges corresponding to this eigenvector. Clearly, collecting all the strong nodal domains and strong nodal edges does not give back the original graph. So, rather than dealing with strong or weak nodal domains, we will consider the decomposition of graphs given by Urschel in \cite{U} defined below.

\begin{figure}[ht]
\centering
\includegraphics[scale=0.2]{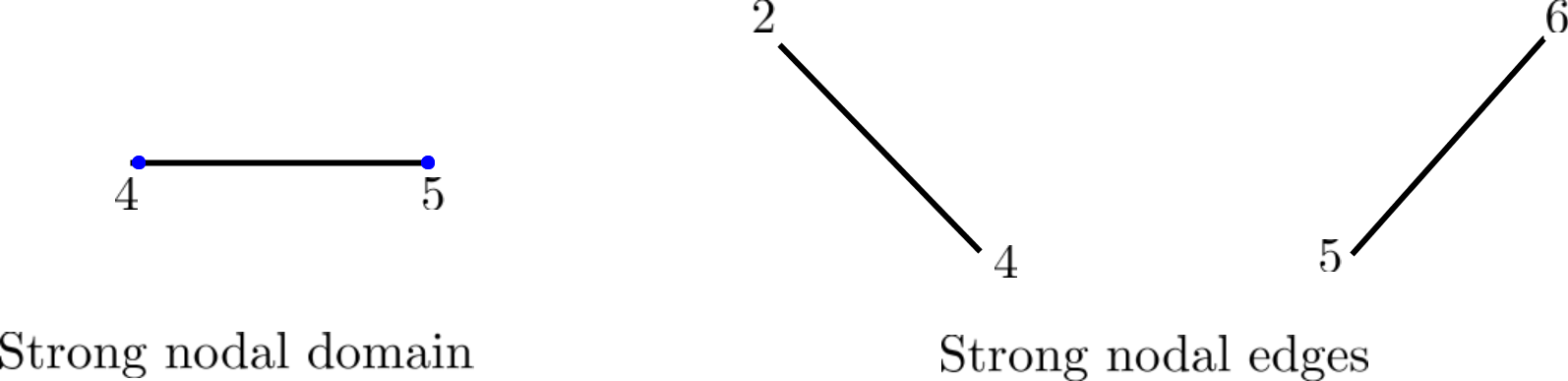}
\caption{Strong nodal domains and nodal edges of the eigenvector $(0, 1, 0, -1, -1, 1, 0)^T$}  %$(0, 1, 0, -1, -1, 1, 0)^T$ corresponding to eigenvalue $\lambda=3$.}
\label{fig: strong nodal}
\end{figure}

%We begin by looking at the nodal decomposition of graphs given by Urschel in \cite{U}.

\begin{defi}[Nodal decomposition \cite{U}]\label{defi: Urschel nodal decomposition}
A nodal decomposition of a graph $G$ with respect to an eigenvector $f$ is a partition
of the vertex set $V$ into 
$\{V_i\}$ with $$V_1 \cup V_2 \cup  \dots \cup V_i= V 
\text{ and }  V_i \cap V_j = \phi   \text{ for all } i \neq  j,$$
such that the subgraphs $G(V_i), i = 1, \dots, s$ are the strong nodal domains of some vector g satisfying
 \[   
g(v) = 
     \begin{cases}
       -1 \text{ or } 1, &\quad\text{if 
$f(v) = 0$ }\\
       f(v), &\quad\text{if $f(v)\neq 0$.}\\
     \end{cases}
\]
The minimum number $s$ for which a nodal decomposition exists is denoted by $\DDD(f)$ and we refer to it as the ``number of nodal decomposition''.
\end{defi}
Coming back to the example above, using the decomposition above  corresponding to $f=(0, 1, 0, -1, -1, 1, 0)^T$,  we have the following.
\vspace{2mm}
 \begin{figure}[h]
\centering
\includegraphics[scale=0.17]{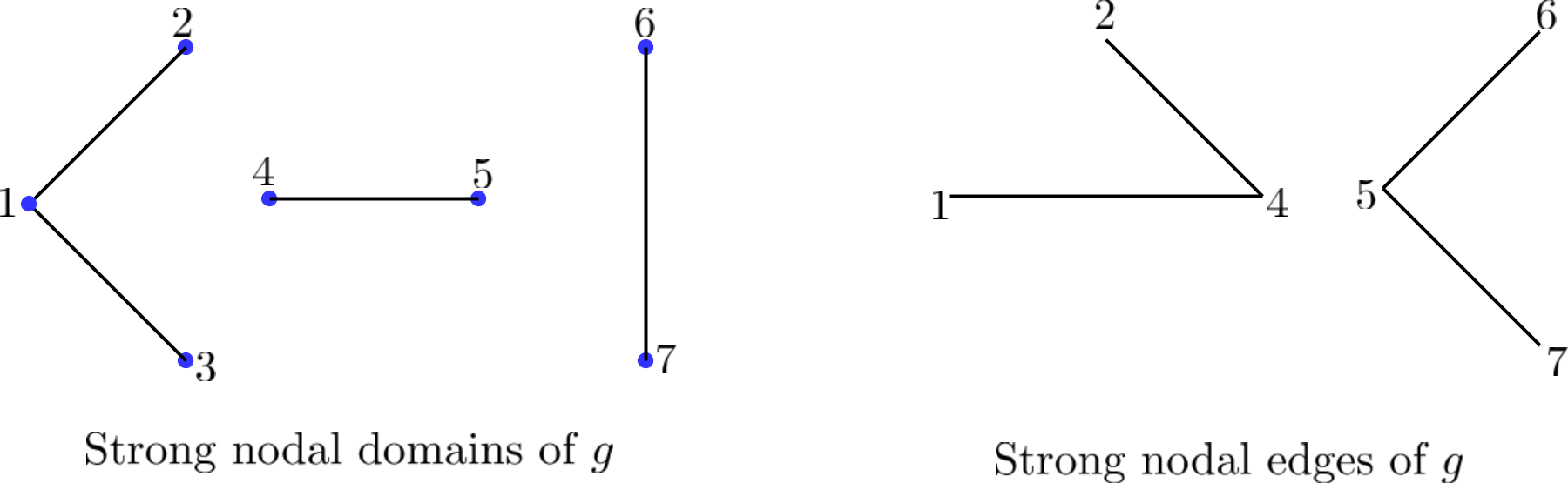}
\caption{Considering $g=(1,1,1,-1,-1,1,1)^T$.}
%\label{fig:Chladni Smithsonian}
\end{figure}

 \begin{figure}[h]
\centering
\includegraphics[scale=0.17]{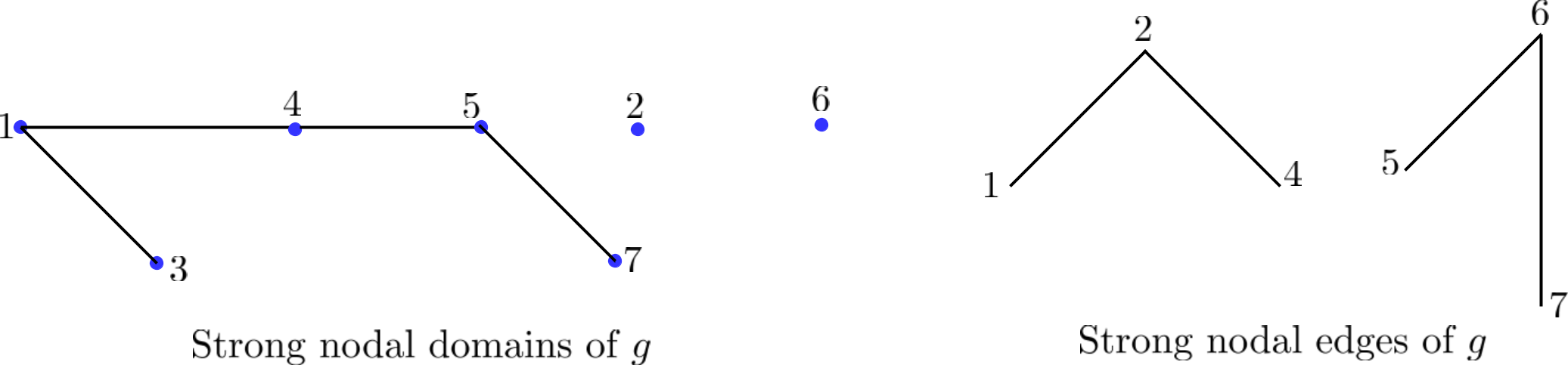}
\caption{Considering $g=(-1,1,-1,-1,-1,1,-1)^T$.}
%\label{fig:Chladni Smithsonian}
\end{figure}

     \vspace{2mm}    
Considering every other possible $g$ corresponding to $f$ in the above example, we see that $\DDD(f)=3$. 

\iffalse
We next give the example of complete graphs $K_n.$ It is easy to calculate the Laplacian for the complete graph. Still, we give a pointwise description below of the Laplacian matrix $L(\PPP(K_n)$ for the sake of completeness.
\[   
L_{ij} = 
     \begin{cases}
       n-1, &\quad\text{if $i=j$}\\
       -1, &\quad\text{if $i \neq  j$} \\
     \end{cases}.
\]
It is well known that for a complete graph on $n$ vertices, all the eigenvalues of the Laplacian except the first equal $n$. And for any nonzero eigenvalue $n$, the corresponding eigenvector $f_n$ is $(0,,\dots,0,1,0,\dots,0,-1)$. Therefore, for the complete graph, clearly there are $2$ weak nodal domains and no strong nodal domains.  Moreover, for all eigenvectors $f_k$, we have $\DDD(f_k) =2.$ This follows by breaking $V$ as $V_1=\{v_1, v_2,\dots, v_{n-1} \}$ and $V_2=\{v_n\}$ and taking $g$ as $g(v_i)=1$ for $1 \leq i \leq n-1$ and $g(v_n)=-1$. 
\fi

Now we mention the following \emph{nodal decomposition theorem} by Urschel, in \cite{U}, which can be observed as a discrete analogue to Courant's nodal domain theorem for the continuous Laplacian.

\begin{theorem}[Urschel]\label{Urschel main theorem}
Let $G=(V,E)$ be a connected graph and $M$ be an associated generalised Laplacian. Then for any eigenvalue $\lambda_k$, there exists a corresponding eigenvector $f_{\lambda_k}$ such that $\mathcal{D}(f_{\lambda_k})\leq k$. Moreover, the set of $f_k$ in the eigenspace $E(\lambda_k)$ with $\DDD(f_{\lambda_k})\leq k$ has co-dimension zero.
\end{theorem}

Now we ask a question similar to the one in Section \ref{sec: intro}, for the case of graph Laplacian. 
\begin{que}\label{question: main question}
Given an eigenvector $f_{\lambda_k}$ of $L(G)$, can we find a constant $\xi$ (depending on $k$ and the number of vertices) for which $\xi\leq \DDD(f_{\lambda_k})$?
\end{que}
In what follows, we show that such a non-trivial constant does not exist. We will construct a class of graphs with a fixed number of vertices $N$ {(however large)}, such that for at least one eigenvector $f_{\lambda}$ corresponding to any non-zero eigenvalue $\lambda$ (except the highest one), $\DDD(f_{\lambda})$ is exactly two. We would like to point out here that the high regularity of complete graphs gives us that all the non-zero eigenvalues are equal which makes studying any eigenvector of a complete graph equivalent to studying the second eigenvector of a complete graph. Looking at the continuous Laplacian case where every second eigenfunction has exactly two nodal domains, it is expected that for second eigenfunction $f$ of a complete graph, $\DDD(f)=2$. This is evident when we put together Theorem \ref{Urschel main theorem} and Proposition \ref{thm: at least two decomposition} below. Using this fact, we would want our graph that addresses Question \ref{question: main question} to be sufficiently irregular. On the other hand, looking at the examples given by Stern and Lewy which have highly repeated spectra,  we would also expect our graph to have some repetition in its spectrum which corresponds to retaining some amount of regularity.

Before describing our main result, we look at some more definitions and notations.

\begin{defi}[Unions of graphs]
The union of graphs $G_1$ and $G_2$, denoted by $G_1 \cup G_2$, is the graph with vertex set $V (G_1 ) \cup V (G_2 )$ and edge set $E(G_1 ) \cup E(G_2 )$. 
\end{defi}

\begin{defi}[Join of two graphs]
    Given two disjoint graphs $G_1$ and $G_2$, the graph formed by connecting each vertex of $G_1$ to each vertex of $G_2$ along with the edges already in $G_1$ and $G_2$ is referred to as the join of $G_1$, $G_2$ and is denoted by $G_1+G_2$.
\end{defi}

 Sabidussi, in \cite{Sa}, generalised the above idea of graph-join as below. 
\begin{defi}[$G$-join]
Let $G=(V, E)$ be a graph with vertex set $V(G)=\{v_1, \cdots, v_n\}$ and $\{G_{v_i}: i=1, \cdots, n\}$ be a collection of graphs indexed by $V(G)$. By the $G$-join of $G_{v_i}'s$ we mean the graph $\GGG:=G(G_{v_i}: v_i\in V(G))$ given by 
$$V(\GGG)=\{(x, y): x\in V(G_y) \text{ and } y\in V(G) \}$$
and 
$$E(\GGG)=\{[(x, y), (x', y')]: [y, y']\in E(G), \text{ or } y=y' \text{ and } [x, x']\in E(G_y) \}.$$
We refer to $G$ as the base graph of $\GGG$.
\end{defi}

To state it plainly, the graph is obtained by replacing each vertex $v_i\in V(G)$ with a graph $G_{v_i}$ and inserting either all or none of the possible edges between vertices of $G_{v_i}$ and $G_{v_j}$ depending of whether $v_i$ and $v_j$ has an edge in $G$ or not.

\begin{defi}[Representation of a graph]
A representation of a graph $\GGG$ is a collection of graphs $R:=\{G; K_{n_y}: y\in V(G)\}$, where $K_{n_y}$ denotes a complete graph with $n_y$ vertices, such that $\GGG \approx G(K_{n_y}: y\in V(G))$. Two representations $R_1=\{G_1; K_{n_y}: y\in V(G_1)\}$ and $R_2=\{G_2; K_{m_z}: z\in V(G_2)\}$ are isomorphic if and only if $G_1 \approx G_2$. $R_1$ is equivalent to $R_2$ if and only if $G_1\approx G_2$ with $\eta: V(G_1)\to V(G_2) $ and $n_y= m_{\eta(y)}$ for all $y\in V(G_1)$. $R$ is trivial if $n_y=1$ for all $y\in V(G)$.
\end{defi}

Now, we state the main result of this article. 

\begin{theorem}\label{thm: main theorem 1}
Let $\GGG$ be any given graph such that $\#V(\GGG)=N$ and $R=\{G; K_{n_y}: y\in V(G)\}$ be a representation of $\GGG$. Let $L(\GGG)$ be the corresponding Laplacian matrix. If $G$ is a complete multipartite graph, then there exists a basis $\BBB$ of $\RR^N$ consisting of eigenvectors of $L(\GGG)$ which satisfies the following property:\\
(SLb): For any eigenvector $f\in \BBB$ corresponding to $\lambda\neq 0$ such that $\lambda$ is not the largest eigenvalue of $L(\GGG)$, we have that $\DDD(f)=2$. 
\end{theorem}

As an immediate application to the above theorem, we have the following
\begin{corollary}\label{cor: sharp lower bound}
    Given any $N\in \NN$ (sufficiently large), there exists a graph with $N$ vertices such that for at least one eigenfunction $f_{\lambda_i}$ corresponding to the eigenvalue $\lambda_i$ ($1<i<N$), we have $\DDD(f_{\lambda_i})=2$.
\end{corollary}

We also prove the following partial result on the nodal decomposition with respect to the eigenvector corresponding to the non-trivial highest eigenvalue (the highest eigenvalue is different from its lower non-zero eigenvalues)  of any graph $G$.
 
\begin{theorem}
\label{thm:nodal-decom-highest-ev}
    Let $G$ be a graph in $n$ vertices, and $G$ has only one dominating vertex\footnote{A vertex $v$ is called a \emph{dominating vertex} if it is adjacent to every other vertex of the graph.}. Suppose that the dominating vertex is also a cut vertex\footnote{A vertex $v$ in a connected graph $G$ is a \emph{cut vertex} if the induced graph after deleting the vertex $v$ is disconnected} of $G$. Let $\lambda_{n}$ denote the highest eigenvalue of $G$. Then $\lambda_n$ is simple and $\DDD(f_{\lambda_{n}}) \geq 3.$
\end{theorem}

Using the above theorems, we end this article with the following classification of cyclic groups in the class of abelian $p$-groups in terms of nodal decomposition of their power graphs (see Definition \ref{defi: power graph} below).

\begin{theorem}\label{thm:nodal_decomp_powergraphs_pgroups}
Let $H$ be a finite abelian $p$-group and $L(\mathcal{P}(H))$ be the Laplacian corresponding to the power graph $\mathcal{P}(H)$. Then, $H$ is cyclic if and only if
for any non-zero eigenvalue $\lambda$ of $L(\mathcal{P}(H))$, there exists a basis $\BBB_{\lambda}$ of the corresponding eigenspace 
$E(\lambda)$ for which $\DDD(f_{\lambda})=2$ for all $f_{\lambda} \in \BBB_{\lambda}.$
\end{theorem}

%
%{\color{blue}
%\begin{remark}
%The above result can be extended to graphs $\GGG$ which satisfy the conditions of the above theorem with $G$ as a complete multi-partite graph. The proof is similar to the proof of the above theorem with some routine modifications, and we skip the details.
%\end{remark}}

\section{Proof of the main results}\label{sec: main result}
Before beginning our proof, we provide the following results, which will be used crucially in the proof. First, we add on to Theorem \ref{Urschel main theorem} with the following simple observation.
\begin{proposition}\label{thm: at least two decomposition}
Let $G=(V, E)$ be a connected graph, and $L$ be the associated Laplacian matrix. Given any non-zero eigenvalue $\lambda_k$, for every eigenvector corresponding to $\lambda_k$, $f\in E(\lambda_k)$, we have $\mathcal{D}(f)\geq 2$. Here, $E(\lambda_k)$ denotes the eigenspace corresponding to $\lambda_k$.
\end{proposition}

\begin{proof}
Given any connected graph $G$ with $n$ vertices, let $L(G)$ be the corresponding Laplacian matrix. We know that the smallest eigenvalue of $L(G)$ is $\lambda_1=0$ since $\det(L(G))=0$. Since the graph is connected, we have that the first eigenvalue is simple. It is easy to verify that $f_0 =(1, 1, \cdots, 1)^T$ is an eigenvector corresponding to $0$. %Then 
%$$E(0)=\{c(1, \cdots, 1): c\in \RR\}.$$

Let $\lambda$ be any non-zero eigenvalue and $f_\lambda=(f_1, \cdots, f_n)^T\in E(\lambda)$ be an eigenvector corresponding to $\lambda$ for which $\mathcal{D}(f_\lambda)=1$. If possible, let $f_k<0$ for some $k\in \{1, \cdots, n\}$. Now, since $f_k<0$, For any vector $g=(g_1, \cdots, g_n)$ defined as
\[   
g_k = 
     \begin{cases}
        1 ~~\text{or}~~ -1, &\quad\text{if 
$f_k = 0$ }\\
       f_k, &\quad\text{if $f_k\neq 0$,}\\
     \end{cases}
\]
we must have at least two strong nodal domains of $g$. This gives us that $\DDD(f_\lambda)$ is at least two, a contradiction.
So, we have that $f_i\geq 0$ for all $i=1, \cdots, n$. Also, note that $f_0$ and $f_\lambda$ are orthogonal. From these two facts, we have
\begin{align*}
    \langle f_0, f_\lambda \rangle= \sum_{i=1}^n f_i = 0 \quad \text{ if and only if } f_i=0 ~~\forall i\in \{1,\cdots, n \}.
\end{align*}
But $f_i=0 ~~\forall i\in \{1,\cdots, n \}$ implies that $f_\lambda\notin E(\lambda)$, a contradiction.
So, $\mathcal{D}(f_\lambda)\geq 2$.

\end{proof}

\begin{lemma}
\label{lem:relation_between_vc_and_nodal_decom}
Let $G$ be a connected graph on $n$ vertices. Let $f$ be an eigenvector of the Laplacian matrix $L(G)$ such that $f$ has exactly one negative component, and the corresponding vertex is not a cut-vertex. Then, $\mathcal{D}(f_{\lambda}) =2$. 
\end{lemma}
More generally, the above lemma is true for any graph $G$ whose vertex connectivity $\kappa(G)>1$. For a graph $G$ with vertex connectivity $\kappa(G)>1$, no vertex is a cut-vertex. Then for any eigenvector $f$ with exactly $1$ negative entry, we can use the above lemma, and we should have $\mathcal{D}(f)=2$.

\begin{proof}
 Let $f=(f_1, \cdots, f_n)$ be an eigenvector of $L(G)$ such that exactly one component is negative. Without loss of generality, consider $f_1$ to be negative. We break $V(\Gamma)$ as the disjoint union of the following two sets:
$W_1=\{v_1 \}$ and $W_2= \{ v_2, \dots,  v_n \}$. We construct $g$ as follows: $g(v_i)=1$ for all those vertices  $v_i$ such that $f(v_i)=0$ and $g(v_i)=f(v_i)$ otherwise. From the definition of $\DDD(f)$, we know that $G$ can be decomposed into $\DDD(f)$ connected subgraphs of $G$. We name the vertex set of each subgraphs as $V_i$, where $i= 1, \cdots, \mathcal{D}(f_\lambda)$. %such that each subgraph $G(V_i)$ is a subgraph of $G$. 

Clearly 
$$W_1= V_1=\{v_1\} \text{(say) \hspace{5pt} and \hspace{5pt}} W_2=\bigcup_{i=2}^{\mathcal{D}(f_\lambda)} \{V_i\}.$$ 
Since $v_1$ is not a cut-vertex, the subgraph formed by vertex set $W_2$ is connected. This implies that $\mathcal{D}(f_{\lambda}) = 2$ which concludes the proof.  
%On the other hand $D(f_{\lambda})=1$ if and only if the eigenvector $f_{\lambda}$ has no negative entry which is a  contradiction. Thus $D(f_{\lambda}) \geq 2$, proving the lemma. 
\end{proof}

\begin{theorem}\label{thm: basis vectors}
Let $\displaystyle \GGG\approx \left(\bigcup_{i=1}^{n_1} K_{p_i} \right)+ \left(\bigcup_{i=n_1+1}^{n_1+n_2} K_{p_i} \right)+\dots+\left(\bigcup_{i=n_1+\dots+n_{s-1}+1}^{n_1+\dots+n_{s-1}+n_s} K_{p_i} \right)$ such that $\#V(\GGG)=N$ and $s>1$. Then there exists a basis $\mathcal{B}$ of $\RR^N$ consisting of eigenvectors such that property $(SLb)$ holds. 
\end{theorem}

\begin{proof}
Note that $\GGG$ is a graph with representation $R=\{G; K_{p_y}: y\in V(G)\}$, where $G$ is a complete $s$-partite graph with $s\geq 2$ and $\#V(G)= N'$. By the definition of $s$-partite graph, let $n_1+\cdots+n_s=N'$ ($n_i\in \NN$) be such that the vertices of $G$ are partitioned into $s$ independent component each with cardinality $n_i~ (i=1, 2, \cdots, s)$. For ease of formulation, we adopt the following notations:

\begin{itemize}
    \item Denote $\displaystyle \sum_{i=1}^{n_1}p_i=N_1$,
$\displaystyle \sum_{i=n_1+1}^{n_1+n_2}p_i=N_2$, \dots, 
$\displaystyle \sum_{i=n_1+n_2+\dots+n_{s-1}+1}^{n_1+n_2+\dots+n_{s}}p_i=N_s$. It is clear that $N_1+N_2+\dots+N_s=N.$

    \item Denote the partial sums as $\sum_{i=1}^t n_i= N'_t$, where $t\in\{1, 2, \dots, s\}$ and $N'_0=0$.
\end{itemize}

We re-enumerate the vertices of $\GGG$ as follows: for $2 \leq l \leq N'$, let
\begin{align*}            
    &V_1= \{0, 1, \cdots, p_1-1\}; V_{l}= \left\{\sum_{i=1}^{l-1}p_i , \left(\sum_{i=1}^{l-1}p_i\right)+1, \cdots, \left(\sum_{i=1}^{l-1}p_i\right)+p_l  -1\right\} 
\end{align*}
and rename the vertex set of $\GGG$ as %$\displaystyle V=\bigcup_{r=1}^{n_1+\dots+n_s} V_r$.
$\displaystyle V=\bigcup_{r=1}^{N'} V_r$. Note that $V_r$ corresponds to the vertex set of $K_{p_r}$.

We now look at the pointwise description of the Laplacian matrix of $\GGG$, $L(\GGG):=(L_{ij})$. We observe that for any $i\in V_{l}~(l\geq 1)$, there exists a unique $r \in \{1, \cdots, s\}$ such that $N'_{r-1}< l \leq N'_r,$ that is, the vertex set $V_l$ corresponds 
to one of the vertices of the independent partitions of $G$ (the $r$-th component) with cardinality $n_r$. We have, for $0\leq i, j\leq N-1$,
    
\[    
    L_{ij} = 
     \begin{cases}
       N-N_r+p_l-1, &\quad\text{if $i=j$}\\
       -1, &\quad\text{if $i\neq j$ and $j \in V_{l} \cup \left( \bigcup_{k=1}^{N'_{r-1}} V_{k} \right)
       \cup 
       \left( \bigcup_{k=N'_{r}+1}^{N'} V_{k}
       \right) $}\\
      0   &\quad\text{if $i\neq j$ and $j \in   \left( \bigcup_{k=N'_{r-1}+1}^{N'_{r}} V_{k}
       \right) \setminus V_l $} \\
     \end{cases}
\]
%\item[(b)] for $i\in V_{r}$ with $n_1+1 \leq r \leq n_1+n_2$
%\[    
%    L_{ij} = 
%     \begin{cases}
%       p_r-1+m, &\quad\text{if $i=j$}\\
  %     -1, &\quad\text{if $i\neq j$ and $j \in \left( \bigcup_{k=1}^{n_1} V_k  \right) \cup V_r $}\\
 %     0   &\quad\text{if $i\neq j$ and $j \in \left( \bigcup_{k=1}^{n_2} V_{n_1+k} \right)\setminus V_r$ } \\
   %  \end{cases}
%\]

%\end{enumerate}

Let $x=(x_0, \cdots, x_{N-1})$ be an eigenvector of $L(\GGG)$ corresponding to an eigenvalue $\lambda$. Therefore we have, $$\lambda x_i= \sum_{j=0}^{N-1} L_{ij}x_j, \quad \text{for } i=0, \cdots, N-1.$$

We now note down the eigenvalues and the corresponding eigenvectors. In this regard, we first see that for any
$1 \leq r \leq s$, $N-N_r$ is an eigenvalue with multiplicity $n_r-1$. For this, we consider the vectors $Z_w=(z_0^w, \cdots, z_{N-1}^w)$ defined as  
\[   
z_q^w = 
     \begin{cases}
       1, &\quad\text{if $q\in V_{N'_{r-1}+w}$}\\
       -\frac{p_{N'_{r-1}+w}}{p_{N'_{r}}}, &\quad\text{if $q \in V_{N'_{r-1}}$} \\
       0, &\quad\text{otherwise}, \\
     \end{cases}
\]
where $w\in \{1, \cdots, n_r-1\}$. Each $Z_w$ is an eigenvector corresponding to the eigenvalue $\lambda= N-N_r$. Then $\mathcal{B}_{N-N_r}= \{Z_w: w=1, \cdots, n_r-1\}$ forms a linearly independent set of eigenvectors corresponding to $\lambda=N-N_r$.

For any $1 \leq r \leq s $ and any $N'_{r-1} < l \leq N'_{r}$, we now show that $N-N_r+p_l$ is also an eigenvalue of $L$ with multiplicity $p_l-1$. To see this, we consider the vectors $X_{w,r}=(x_0^{w,r}, \cdots, x_{N-1}^{w,r})$ defined as 
\[   
x_q^{w,r} = 
     \begin{cases}
       1, &\quad\text{if $q=p_1+p_2+\dots+p_l-1-w$}\\
       -1, &\quad\text{if $q=p_1+p_2+\dots+p_l-1$} \\
       0, &\quad\text{otherwise}, \\
     \end{cases}
\]
for each $w\in \{1, \cdots, p_l-1\}$. Here $\mathcal{B}_{N-N_r+p_l}= \{X_{w,r}: w=1, \cdots, p_l-1\}$ forms a linearly independent set of eigenvectors corresponding to $\lambda=N-N_r+p_l$.

Finally, we look at the highest and lowest eigenvalues of $L(\GGG)$. For the vectors $Y_w=(y_0^w, y_1^w, \dots, y_{N-1}^w)$ defined as follows:

\[   
y_q^w= 
     \begin{cases}
       1, &\quad\text{if $q \in V_{N'_{w-1}+1} \cup V_{N'_{w-1}+2}   \cup \dots \cup V_{N'_{w}} $}\\
       -\frac{N_w}{N_s}, &\quad\text{if $q \in
       V_{N'_{s-1}+1} \cup V_{N'_{s-1}+2} \cup \dots \cup V_{N'}$}  \\
     \end{cases}
\]
where $w \in  \{1, 2 ,\dots, s-1\},$ we have that $Y_w$ is an eigenvector corresponding to the eigenvalue $N$ ($1 \leq w \leq s-1$). The eigenvectors are also clearly independent, which gives us that the eigenvalue $N$ is of multiplicity $s-1$. Furthermore, it is known that $0$ is a simple eigenvalue of $L(\GGG)$ with eigenvector $\BBB_0=\{(1, 1, \cdots, 1)\}$. 

Note that
\[
1 + (s-1)+\sum_{r=1}^s (n_r-1) + \sum_{r=1}^s \sum_{l=N'_{r-1}+1}^{N'_r} (p_l-1)  =s +\sum_{r=1}^s (n_r-1) + \sum_{r=1}^s (N_r-n_r)  = N.
\]

In the above equation, we see that by adding the multiplicities of all the above eigenvalues, we get back $N$, which tells us that we have all the possible eigenvalues and a basis with eigenvectors of $\GGG$. 

In order to show that this basis of $\GGG$ satisfies $(SLb)$, we look at the following proof of Theorem \ref{thm: main theorem 1}.
\end{proof}

 Using the eigenvectors of $L(\GGG)$ we found above, to prove Theorem \ref{thm: main theorem 1}, our remaining work is to show that for every $f\in\mathcal{B}_{\lambda_i}$ such that $\lambda_i$ is neither the maximum nor the minimum, $\mathcal{D}(f)=2$. 

\begin{proof}[Proof of Theorem \ref{thm: main theorem 1}]
We follow the notations from Theorem \ref{thm: basis vectors} for the proof. For the eigenvalues of the form $N-N_r+p_l$ (for some $1\leq r \leq s$ and $N'_{r-1}< l\leq N'_r$),  we see that, for each $f\in \mathcal{B}_{\lambda_i}$, there is exactly one negative component. We first observe that if $s \geq 3$ then the vertex connectivity is always greater than 1. If $s=2$, then the vertex with the negative entry is a cut vertex if and only if $N_r=1$ and hence $p_l$ must be $1$. But in that case, $N-N_r+p_l$ is the same as $N$, which is the highest eigenvalue. So, we can ignore this case since we are looking at only the eigenvectors corresponding to eigenvalues that are neither the highest nor the lowest. We can now use Lemma \ref{lem:relation_between_vc_and_nodal_decom} to get that $\mathcal{D}(f)= 2$. 
%If any $p_l=1$, then $N-m_r+p_l$ has multiplicity $0$ and not even an eigenvalue. 
%If $m=1$ then though there is one negative component, but the corresponding vertex is not a cut vertex and hence using Lemma \ref{lem:relation_between_vc_and_nodal_decom}
%we are done. 
Thus, for the eigenvalues 
$\lambda= N-N_r+p_l$, we have $\mathcal{D}(f)=2$.

We now consider the eigenvalues $\lambda=N-N_r$ (for some $1\leq r \leq s$). As the multiplicity of the eigenvalues $\lambda=N-N_r$ is $n_r-1$, for $\lambda=N-N_r$ to be an eigenvalue, we must have $n_r>1$. 
Given any $f=(f_0, \cdots, f_{N-1})\in \mathcal{B}_{\lambda}$ we observe that, $f_q <0$ for all $q\in V_{N'_r}$ and $f_q\geq 0$ for all $q\notin V_{N'_r}$. We now consider the vector $g=(g_0, \cdots, g_{N-1})$ such that
\[   
g_q= 
     \begin{cases}
       1, &\quad\text{if $f_q=0$ }\\
       f_q, &\quad\text{if $f_q \neq 0$}, 
     \end{cases}
\]
We see that $K_{V_{N'_r}}$ and $\GGG-K_{ V_{N'_r}}$   are two connected subgraphs of $\GGG$. Being a complete graph, $K_{V_{N'_r}}$ is connected and since $n_r>1$, we have that $\GGG-K_{ V_{N'_r}}$
is connected. Therefore, we have $\DDD(f)=2$. 

This completes the proof. 
\end{proof} 

The proof of Corollary \ref{cor: sharp lower bound} follows directly from Theorem \ref{thm: main theorem 1}. Note that the representation of the graph given below is not unique and the lower bound on $N$ is not optimal either. But since we are interested in relatively large graphs, the following graph serves the purpose.

\begin{proof}
[Proof of Corollary \ref{cor: sharp lower bound}]
Let $N\in \NN$ with $N>19$ be any given number. Consider the graph 
$\displaystyle \GGG= \left(K_2 \cup K_3 \right)   +  \left(K_4+K_{N-9} \right)$. 
Since $N>19$, we have that the number of vertices in $K_{N-9}$ is greater than the total number of vertices in the remaining components. This combined with the fact that $G$ (the base graph of $\GGG$) is a bi-partite graph with $n_2=2$ implies that the highest eigenvalue $\lambda_N$ of $\GGG$ is simple. Since we are interested in constructing a graph such that $\DDD(f_{\lambda_i})=2$  for at least one eigenfunction $f_{\lambda_i}$ with $1<i<N$, we can ignore the eigenfunction corresponding to $\lambda_N$. %Using this fact, we now use Theorem \ref{thm: main theorem 1} to get that the multiplicity of the highest of $\GGG$  is $1$. 
The result now follows from Theorem \ref{thm: main theorem 1}.
\end{proof}

\iffalse
\begin{proof}
%\left(\bigcup_{i=t+1}^{n}$$ K_{n_1}\right)$
%and  $P=(p_1, \cdots, p_{n})$ be any partition of $N$ such that $p_1\leq p_2\leq \dots \leq p_n$ ($p_{n}>1$ ???? or $n>3$????), that is, $\displaystyle \sum_{i=1}^{n} p_i=N$. Fix any number $1< t < n$ and consider the graph $\displaystyle \GGG= \left(\bigcup_{i=1}^{t} K_{p_i}\right)+ \left(\bigcup_{i=t+1}^{n} K_{n_1}\right)$. Then $\GGG$ has the representation $R=\{G; K_{p_i}: i\in {1, \cdots n}\}$, where $G$ is a bipartite graph with the vertex partition $\{1, \cdots, t\}\cup \{t+1, \cdots, n\}$. 
Then the corollary follows from Theorem \ref{thm: main theorem 1}.
\end{proof}
\fi
%For each $r\in \{1, \cdots, n_1\}$ we have that $\mathcal{B}_{n+p_r}= \{X_s^r: s=1, \cdots, p_r-1\}$ is a linearly independent collection of eigenvectors corresponding to eigenvalue $n+p_r$ and for every $r\in \{n_1+1, \cdots, n_1+n_2\}$,  
%Thus, $\mathcal{B}_{m+p_r}= \{X_q^r: s=1, \cdots, p_r-1\}$ is a linearly independent collection of eigenvectors corresponding to eigenvalue $m+p_r$.

%\begin{enumerate}

    %\item[(a)] for $i\in V_{r}$ with $1 \leq r \leq n_1$, we have
    %\begin{equation}\label{eq: eigen equation for i in R bipartite}
    %    \lambda x_i= (p_r-1+n)x_i - \sum_{j\in V_r \cup V_{n_1+1} \cup \dots \cup V_{n_1+n_2}, i\neq j} x_j.
    %\end{equation}
    
    %\item[(b)] for $i\in V_r$ with $n_1+1 \leq r \leq n_2$, we have
   
   %\begin{equation}\label{eq: eigen equation for i in P bipartite}
   %    \lambda x_i= (p_r-1+m)x_i - \sum_{j\in V_{1}\cup V_{2}, \dots V_{n_1} \cup V_r, i\neq j} x_j .
   %\end{equation}
    
%\end{enumerate}

Before moving forward with the proof of Theorem \ref{thm:nodal-decom-highest-ev}, we mention the following results of Mohar in \cite{Mo} regarding the Laplacian spectrum of a graph.

\begin{theorem}[Mohar]
\label{thm:mohar_1} Let $G$ be a graph with $n$ vertices and $\overline{G}$ denote the complement\footnote{The complement of a graph $G$ is a graph $H$ with the same vertices such that two distinct vertices of $H$ are adjacent if and only if they are not adjacent in $G.$} of $G$. Then $\lambda_n (G ) \leq n$ and equality holds if and only if $\overline{G}$ is not connected.
\end{theorem}

In \cite{Mo}, Mohar also proved the following which provides a formulation for the Laplacian spectrum of the join of two graphs. For a graph $G$, let $\Theta(G,x)$ denote the characteristic polynomial of $L(G)$.

\begin{theorem}[Mohar] 
\label{thm:mohar_2}
Let $G_1$ and $G_2$ be disjoint graphs with $n_1$ and $n_2$ vertices, respectively.  Then, 
$$ \Theta (G_1 + G_2 , x) =  \frac{x(x-n_1-n_2)}{(x-n_1)(x-n_2)} \Theta (G_1, x-n_2) \Theta (G_2, x-n_1).$$
\end{theorem}

\begin{proof}[Proof of Theorem \ref{thm:nodal-decom-highest-ev}]
Let $v$ be the cut vertex which is also a dominating vertex from our assumption. Then, the graph $\overline{G}$ is  disconnected. Therefore, by Theorem \ref{thm:mohar_1}, we have $\lambda_n=n.$  We now find the corresponding eigenvector. %is 
%\[f_{\lambda_n}=(x_1,x_2,\dots,x_n)=(n-1,-1,-1,\ldots,-1).\] 
Without loss of generality, we assume that the vertex $v$ corresponds to the first row and column of $L(G)$. Since $v$ is a dominating vertex, even though we do not have the complete pointwise form for $L(G)$, we have the following information:
\begin{enumerate}
    \item $L_{11}=n-1 $ and $L_{1j}=-1$ for all $ j \neq 1.$
    \item For all $i>1$, we have $L_{i1}=-1$. Moreover, for any $L_{ii}=r$, here are $r-1$ indices, say $j_1, j_2, \dots, j_{r-1} > 1$ such that $L_{ij_1} =L_{ij_2}=\dots=L_{ij_{r-1}}= -1.$
\end{enumerate}
Using the above information, we have 
\[\sum_{j=1}^n L_{1j} x_j = (n-1)(n-1)+(-1)(-1)+\dots+(-1)(-1) 
 = n(n-1)\]
and for $1<i\leq n$, we have 
\[\sum_{j=1}^n L_{ij} x_j = (-1)(n-1)+(r-1)(-1)(-1)+r(-1)=-n.\]
Combining the above two equations, we have 
\[
\sum_{j=1}^n L_{ij} x_j = n x_i, \hspace{5pt} \text{for } 1\leq i\leq n,
\]
where $(x_1, \dots, x_n)= (n-1, -1, \cdots, -1)$.
Thus, $n$ is an eigenvalue of the Laplacian matrix with the corresponding eigenvector $f_{\lambda_n}=(n-1,-1,-1,\ldots,-1)$. 

We next prove that the eigenvalue $n$ is indeed simple. Let $G'$ be the induced graph on the remaining $n-1$ vertices after deleting the vertex $v$. Clearly, $G'$ has no dominating vertex.  By Theorem \ref{thm:mohar_2}, we then have,

\begin{eqnarray}
\Theta (G, x)=\Theta (K_1 + G', x) & = & \frac{x(x-n)}{(x-1)(x-(n-1))} \Theta (K_1, x-(n-1)) \Theta (G', x-1) \nonumber \\
& = & \frac{x(x-n)}{(x-1)} \Theta (G', x-1)  \label{eqn:vital}
\end{eqnarray} 

As $v$ is a dominating vertex which is also a cut vertex, the graph $G'$ is not connected, which implies that $\overline{G'}$ is connected. Again, using Theorem \ref{thm:mohar_1}, we have that  $n-1$ is not a root of  $ \Theta (G', x-1)$. From (\ref{eqn:vital}), we note that, $n$ is a repeated root of $\Theta(G, x)$ if and only if $n-1$ is a root of $\Theta(G', x-1)$.
This proves that $n$ is a simple eigenvalue of $L(G)$.

We now look at the nodal decomposition of the eigenvector $(n-1,-1,\dots,-1).$ 
The vertex $v$ being a cut vertex, the number of strong nodal domains of $(n-1,-1,\dots,-1)$ has to be the same as the number of connected components of $\overline{G}.$ 
Thus, we  have $\DDD(f_{\lambda_n})>2$ and as the eigenvalue $n$ is simple,
$\DDD(f_{\lambda})>2$ for all $f_{\lambda} \in \mathcal{B}_{\lambda_n}.$
This completes the proof.
\end{proof} 

\begin{remark}
    In Theorem \ref{thm:nodal-decom-highest-ev}, since $f_{\lambda_n}$ does not contain any zero component, we have $\DDD(f_{\lambda_n})=S(f_{\lambda_n})=W(f_{\lambda_n})$. This gives us that the total number of strong nodal domains, $S(f_{\lambda_n})$ should be strictly greater than 2.
\end{remark}

\section{Applications to power graphs}\label{sec: power graph}

The study of graphs arising from various groups has been a topic of increasing interest over the last two decades. The advantage of studying these graphs is multi-fold as they  help us to (1) characterize the resulting graphs, (2) characterise the algebraic structures with isomorphic graphs, and also (3) realize the interplay between the algebraic structures and the corresponding graphs. Many different types of graphs, specifically
power graphs, commuting graphs, enhanced power graphs, etc. have been introduced to explore the properties of algebraic structures
using graph theory. The concept of a power graph was introduced  by Kelarev and Quinn in the context of semigroup theory \cite{KQ} (also see \cite{CGS}).

\begin{defi}\label{defi: power graph}
Given any group $H$, the power graph of $H$ denoted by $\mathcal{P}(H)$ is the graph whose vertices are the elements of $H$ and two vertices $x$ and $y$ are adjacent if $x= y^a$ or $y=x^b$ for some $a,b \in \mathbb{N}$.
\end{defi} 

\iffalse
Various properties of the  power graphs of finite groups have been studied in detail. Aalipour et al. 
characterized  finite groups for which the power and the enhanced power graphs are equal.
Besides, Zahirovic et al. proved that finite groups with isomorphic enhanced power graphs have isomorphic directed power graphs.
\fi

In the last decade, many researchers have studied various spectral properties related to the power graphs of finite groups. Chattopadhyay and Panigrahi \cite{CP}
studied the Laplacian spectra of power
graphs of finite cyclic groups as well as the dihedral groups. Mehranian et al. \cite{MGA} computed the  adjacency spectrum of the power graph of cyclic groups, dihedral groups and elementary abelian groups of prime power order.
Hamzeh and Ashrafi \cite{HA} investigated adjacency and Laplacian spectra of power
graphs of the cyclic and quaternion groups. 
%As an application of our results, we mention
%When the group is abelian, for ease of understanding, we write in additive sense and in this case, two vertices $x$ and $y$ are adjacent when $x=a y$ or $y=b x$ for some $a,b \in \mathbb{N}$.
Using Theorems \ref{thm: main theorem 1} and \ref{thm:nodal-decom-highest-ev}, 
we have an interesting characterisation of certain groups in terms of their nodal decomposition.  As an immediate application, we have the following
\begin{theorem}
\label{thm:nodal_decomp_powergraphs_cyclicgroupsoforderpq}
Let $H$ be a cyclic group of order $pq$ where $p$ and $q$ are distinct primes with $p<q$ and $L(\mathcal{P}(H))$ denotes the Laplacian. For any non-zero eigenvalue $\lambda$ (apart from the highest) of $L(\mathcal{P}(H))$, there exists a basis $\BBB_{\lambda}$ of the corresponding eigenspace 
$E(\lambda)$ for which $\DDD(f_{\lambda})=2$ for all $f_{\lambda} \in \BBB_{\lambda}.$

%Given $p, q$ distinct primes such that $p$ divides $q-1$, if there exists a non-abelian group $G$ of order $pq$ then for the highest eigenvalue $\lambda=pq$ of $L(\mathcal{P}(G))$, we have $\DDD(f_{\lambda})=p+2$ for all $f \in E(pq).$

Let $H$ be the unique non-abelian group of order $pq$  where $p$ and $q$ are distinct primes, and $p$ divides $q-1$. For the highest eigenvalue $\lambda=pq$, we have $\DDD(f_{\lambda})=p+2$ for all $f_{\lambda} \in E(pq).$
\end{theorem}

%%%\begin{remark}
%If $p$ does not divide $q-1$, any group of order $pq$ is cyclic.    
%\end{remark}

\begin{proof}
 When $H$ is a cyclic group of order $pq$, using \cite[Theorem 5]{CS}, we have $\mathcal{P}(H)= (K_{p-1}\cup K_{q-1})+K_{\phi(pq)+1}.$ Thus, by Theorem \ref{thm: main theorem 1}, we are done. 

 When $H$ is non-cyclic, the number of $p$-Sylow subgroups
of $H$ is clearly $q$. Moreover, $H$ also has a unique $q$-Sylow subgroup. Thus the identity element is the only dominating vertex. Moreover, it is also a cut-vertex as any element of order $p$ can never be connected with an element of order $q$ and hence the total number of connected components of $\overline{\mathcal{P}(H))}$ is $p+2.$ The proof now follows from Theorem \ref{thm:nodal-decom-highest-ev}.   
\end{proof}

Finally, we look at the proof of the characterisation of cyclic groups among finite abelian $p$-groups.

\begin{proof}[Proof of Theorem \ref{thm:nodal_decomp_powergraphs_pgroups}]
 By \cite[Theorems 14 and 15]{Pa}, if $H$ is non-cyclic, the vertex connectivity is $1$ which implies that the identity is a cut vertex of $\mathcal{P}(H)$. Hence, if we consider the highest eigenvalue $\lambda$, then the corresponding eigenvector $f(\lambda)$ has $\DDD(f(\lambda))>2$ by Theorem \ref{thm:nodal-decom-highest-ev}.
 
 If $H$ is a cyclic $p$-group, then $\mathcal{P}(H)$ must be complete. Combining Theorem \ref{Urschel main theorem} and Proposition \ref{thm: at least two decomposition}, there exists an eigenvector $f$ corresponding to every non-zero eigenvalue (with repetition) that has $\DDD(f)=2$. This completes the proof. 
\end{proof}

\subsection{Acknowledgements} The first named author acknowledges the Science and Engineering Research Board, India (File No. PDF/2021/001899) for funding this research. The first-named author also wishes to thank the Indian Institute of Science	Bangalore for providing ideal working conditions during the preparation of this work. The second named author would like to thank Iowa State University for providing great working conditions and funding for the research. The initial phase of the project started during their time at the Indian Institute of Technology Bombay, and both authors thank the institute for providing ideal working conditions. Finally, both authors would like to express their gratitude to Mayukh Mukherjee and Gabriel Khan for their insightful comments and suggestions, which substantially improved the article.

%\subsection*{Conflict of interest} On behalf of all authors, the corresponding author states that there is no conflict of interest.

%\subsection*{Data availability} Data sharing does not apply to this article as no datasets were generated or analysed during the current study.

\end{document}